\documentclass[a4paper,12pt]{amsart}

\usepackage{amsthm}
\usepackage{amsmath}
\usepackage{amssymb}
\usepackage{latexsym}
\usepackage{enumerate}

\renewcommand{\thetheoremName}

\newtheorem{proposition[[]]}[theoremName]{Proposition G}

\newtheorem{theorem}{Theorem}[section]
\newtheorem{lemma}[theorem]{Lemma}

\newtheorem{proposition}[theorem]{Proposition}
\newtheorem{corollary}[theorem]{Corollary}

\theoremstyle{definition}
\newtheorem{definition}[theorem]{Definition}

\newtheorem{remark}{Remark}

\numberwithin{equation}{section}

\newcommand{\dist}{\operatorname{dist}}
\newcommand{\Vol}{\operatorname{Vol}}

\newcommand{\erre}{\mathbb{R}}

\newcommand{\LL}{\operatorname{L}}

\numberwithin{equation}{section}

\newcommand{\tors}{\operatorname{\mathcal{A}_{1}}}

\begin{document}

\title[Torsional Rigidity]
{TORSIONAL RIGIDITY OF SUBMANIFOLDS \\ WITH CONTROLLED GEOMETRY}

\author[A. Hurtado]{A. Hurtado$^{\natural}$}
\address{Departament de Matem\`{a}tiques, Universitat Jaume I, E-12071 Castell\'o,
Spain.}
 \email{ahurtado@mat.uji.es}
\author[S. Markvorsen]{S. Markvorsen$^{\#}$}
\address{Department of Mathematics, Technical University of Denmark,  DK-2800 Kgs. Lyngby, Denmark}
\email{S.Markvorsen@mat.dtu.dk}
\author[V. Palmer]{V. Palmer*}
\address{Departament de Matem\`{a}tiques, Universitat Jaume I, Castell\'o,
Spain.} \email{palmer@mat.uji.es}
\thanks{$^{\#}$ Supported by the Danish Natural Science Research Council and the Spanish MEC-DGI grant
MTM2007-62344.\\
\indent * Supported by Spanish MEC-DGI grant No.MTM2007-62344, the Caixa Castell\'o Foundation and a grant of the
Spanish MEC {\em Programa de Estancias de profesores e investigadores
espa\~noles en centros de ense\~nanza superior e investigaci\'on extranjeros.}\\
\indent $^{\natural}$ Supported by Spanish MEC-DGI grant
No.MTM2007-62344, the Caixa Castell\'o Foundation}


%


\begin{abstract}
We prove explicit upper and lower bounds for the torsional
ri\-gi\-di\-ty of extrinsic domains of  submanifolds  $P^m$ with controlled radial mean curvature in
ambient Riemannian manifolds $N^n$ with a pole $p$ and with sectional curvatures bounded from above and from below, respectively. These  bounds
are given in terms of the torsional rigidities of corresponding
Schwarz symmetrizations of the domains in warped product model
spaces. Our main results are obtained using methods from previously established
isoperimetric inequalities, as found in e.g. \cite{MP4} and \cite{MP5}.
As in \cite{MP4} we also characterize the geometry of those situations in which the bounds for the torsional rigidity are actually attained and study the behavior at infinity of the so-called {\em geometric average} of the mean exit time for Brownian motion.
\end{abstract}

\maketitle

\section{Introduction}\label{secIntro}
\bigskip

Given a precompact domain $D$ in a complete Riemannian manifold $(M^{n},g)$, the {\em torsional rigidity of $D$} is defined as the integral

 \begin{equation}\label{eqApk}
    \mathcal{A}_{1}(D)= \int_D E(x) \, d\sigma
    \quad ,
    \end{equation}
    where $E$ is the smooth solution of the Dirichlet--Poisson equation
    \begin{equation}  \label{eqmoments1}
\begin{aligned}
\Delta^{M} E+1 &= 0\,\,\, \text{on}\,\,\, D\\
E\vert_{\partial D} &=0 \quad .
\end{aligned}
\end{equation}

 Here  $\Delta^M$ denotes the Laplace-Beltrami operator on $\,(M^{n},
g)\,$.
    The function $E(x)$ represents the mean time of first exit from $D$
    for a Brownian particle starting at the point $x$ in $D$, see
    \cite{Dy}.\\

    The name {\em torsional rigidity of $D$} stems from the fact that  if $D \subseteq \erre^2$,
    then $\mathcal{A}_1(D)$ represents the torque required per unit angle of twist and per unit length
    when twisting an elastic beam of
    uniform cross section $D$, see \cite{Ba} and  \cite{PS}.
    The torsional rigidity $\mathcal{A}_1(D)$  plays a r\^{o}le in the so-called $L^p$-moment spectrum
    for the domain $D$,
    see \cite{Mc}, which is similar to the r\^{o}le of the first positive Dirichlet eigenvalue,
    the {\em fundamental tone},
    in the Dirichlet spectrum for the domain $D$.\\

    As in  \cite {MP4} we consider a Saint-Venant type problem, namely, how to optimize the torsional rigidity among
    all the domains having the same given volume in a given space or in some otherwise fixed geometrical setting.
    Here we restrict ourselves to a particular class of subsets, namely  the extrinsic balls $D_R$ of a submanifold $P$
    immersed with controlled mean curvature into an ambient manifold with suitably bounded sectional curvatures.\\

    The proof of the Saint-Venant conjecture in the general context of Riemannian geometry makes use of the concept of
    {\em Schwarz--symmetrization} and like the Rayleigh conjecture concerning the fundamental tone it also hinges upon
    the proof of the Faber--Krahn inequality, which in turn is based on isoperimetric inequalities satisfied by the
    domains in question.\\

    Under extrinsic curvature restrictions on the submanifold and intrinsic curvature restrictions on the ambient
    manifold we show in Theorem \ref{thmIsopGeneral1} that the extrinsic balls  satisfy strong isoperimetric
    inequalities, specifically  lower and upper bounds for the $\infty$-isoperimetric quotient
    $\Vol(\partial D_R)/\Vol(D_R)$, where the bounds are given by corresponding $\infty$-isoperimetric
    quotients of certain geodesic balls in tailor-made warped product spaces.\\

    As in \cite{Pa2}, \cite{Ma1}, and \cite{MP4}, the comparison is obtained essentially by transplanting
    the radial solution of a Poisson equation defined in the radially symmetric model space from that model
    to the extrinsic $R$-balls $D_R$ in the submanifold $P$.\\

    Once we have this isoperimetric information at hand, we then apply it to get bounds for the torsional rigidity
    of the extrinsic balls. In \cite{MP4}, following these ideas and inspired by the work \cite{Mc}, we obtained
    upper bounds for the torsional rigidity of the extrinsic balls in a {\em minimal} submanifold of an ambient
    manifolds with sectional curvatures bounded from above.\\

    One key result on the way to upper and lower bounds for the torsional rigidity is Theorem \ref{eqschwarz},
    which shows a fundamental equality between the integral of the transplanted radial solution of the Poisson
    equation in $D_R$ and the corresponding integral of its Schwarz--symmetrization in the model space.\\

    As a consequence of the isoperimetric inequalities in Theorem \ref{thmIsopGeneral1} and the Schwarz symmetrization
    identity in Theorem \ref{eqschwarz}, we obtain {\em lower} and {\em upper} bounds for the torsional rigidity of
    the extrinsic balls in submanifolds of much more general type than the special minimal submanifolds alluded to
    above. In this general setting we only assume that the submanifold has {\em controlled} mean curvature and that
    the ambient manifolds have radial sectional curvatures bounded from below (Theorem \ref{thm2.1}) or from above
    (Theorem \ref{thm2.2}), respectively.\\

   In the work \cite{BBC} the existence of regions in $\erre^m$ with finite torsional rigidity and yet infinite
   volume were considered. To get to such regions, the authors assume Hardy inequalities for these domains. The
   geometric effect of this assumption is to make the volume of the boundary of the regions relatively large in
   comparison with the enclosed volume. In consequence, the Brownian diffusion process finds sufficient outlet-volume
   to escape at the boundary, giving in consequence a small mean exit time and at the same time a small incomplete
   integral of the mean exit time, i.e. a bounded torsional rigidity.\\

   Inspired by this result, the study of the  behaviour at infinity of the geometric average of the mean exit
   time for Brownian motion was initiated in \cite{MP4}. Specifically, given the quotient
   $\mathcal{A}_1(D_R)/\Vol(D_R)$, we may consider the limit of this quotient for $R \to \infty$ as a measure of
   the volume-relative swiftness (at infinity) of the Brownian motion defined on the entire submanifold. It was proved
   in \cite{MP4} that this quotient is unbounded for geodesic balls in all Euclidean spaces as
   $R \longrightarrow \infty$, while it is bounded for geodesic balls in simply connected space-forms of constant
   negative curvature. \\

   The Brownian motion of a particle from a point on a given manifold
is {\it recurrent} if the particle is sure to visit every open set
in the manifold as time goes by. The manifold is thence itself
called recurrent, because the property turns out to be independent
of the starting point. The simplest examples of recurrent
manifolds are the $1-$ and $2-$dimensional Euclidean spaces. If
the 'visiting' condition is not satisfied, then $M^{m}$ is called
{\it transient}. The $n-$dimensional Euclidean spaces are for $n\geq3$ the prime
examples of transient manifolds. We refer to \cite{MP2, MP3, MP5} for results concerning
general transience conditions for submanifolds. We note that transience is
not in itself sufficient to give finiteness of the
geometric average of the mean exit time, as is exemplified by $\mathbb{R}^{n}$ for all $n \geq 3$.\\

   In this paper, we establish  a set of curvature restrictions that do guarantee the finiteness of the average mean
   exit time at infinity, meaning that the Brownian diffusion process is moving relatively fast to infinity (see
   Corollary \ref{cor2}), and a {\em dual} version of this result, that is a set of curvature restrictions which
   guarantee in turn that the average mean exit time at infinity is infinite so that the diffusion is moving
   relatively slow to infinity (see Corollary \ref{cor1}).\\

Throughout this paper we assume that the ambient manifold
$N^n$ possesses a pole $p$ and that $N$ has its $p$-radial sectional curvatures $K_{p,N}(x)$ bounded
from below or from above, respectively, by the expression $-w''(r(x))/w(r(x))$,
which is precisely the formula for the radial sectional curvatures of a so-called
{\em $w$-model space} $\,M^{m}_{w}\,$. Such a model space is
defined as the warped product of a real interval with the standard
$(m-1)$-dimensional unit Euclidean sphere $S^{m-1}_{1}$ and
warping function $w(r)$ which satisfies the initial conditions
$w(0) = 0$, $w'(0) = 1$. A precise definition as well as further instrumental
properties of model spaces will be given in Section \ref{pre}
below.

\subsection*{Outline of the paper}
Section 2 is devoted to the precise definitions of extrinsic balls, the warped product spaces that we use as
models and to the description of the general set-up of our comparison analysis: the comparison constellations. In
sections 3 and 4 we formulate the isoperimetric inequalities and the integral equalities for the Schwarz--symmetrization
of the solution of the Poisson equation, respectively. The main comparison results for the Torsional Rigidity are
stated and proved in Section 5, and finally, in sections 6 and 7 we present an intrinsic analysis of these results
and consider the behavior of the {\em averaged mean exit time} at infinity, respectively.

\subsection*{Acknowledgements}
This work has been partially done during the stay of the third named author at the Department of Mathematics at
the Technical University of Denmark and at the Max Planck Institut f\"{u}r Mathematik in Bonn,  where he enjoyed
part of a sabbatical leave, funded by a grant of the Spanish Ministerio de Educaci\'on y Ciencia. He would like to
thank these institutions for their support during this period and to thank the staff of the Mathematics Department
at DTU and the MPIM for their cordial hospitality.

\section{Preliminaries and Comparison Setting}\label{pre}

We first consider a few conditions and concepts that will
be instrumental for establishing our results.

\subsection{The extrinsic balls and the curvature bounds} \label{subsecurvature}

We consider an immersed $m$-dimensional submanifold $P^m$ in a
complete Riemannian manifold $N^n$. Let $p$ denote a point in $P$
and assume that $p$ is a pole of the ambient manifold $N$. We
denote the distance function from $p$ in $N^{n}$ by $r(x) =
\dist_{N}(p, x)$ for all $x \in N$. Since $p$ is a pole there is -
by definition - a unique geodesic from $x$ to $p$ which realizes
the distance $r(x)$. We also denote by $r$ the restriction
$r\vert_P: P\longrightarrow \erre_{+} \cup \{0\}$. This
restriction is then called the extrinsic distance function from
$p$ in $P^m$. The corresponding extrinsic metric balls of
(sufficiently large) radius $R$ and center $p$ are denoted by
$D_R(p) \subseteq P$ and defined as any connected component which
contains $p$ of the set:
$$D_{R}(p) = B_{R}(p) \cap P =\{x\in P \,|\, r(x)< R\} \quad ,$$
where $B_{R}(p)$ denotes the geodesic $R$-ball around the pole $p$ in $N^n$. The
 extrinsic ball $D_R(p)$ is a connected domain in $P^m$, with
boundary $\partial D_{R}(p)$. Since $P^{m}$ is assumed to be unbounded in $N$ we
have for every sufficiently large $R$ that $B_{R}(p) \cap P \neq P$.
\bigskip

We now present the curvature restrictions which constitute the geometric framework of our investigations.

\begin{definition}
Let $p$ be a point in a Riemannian manifold $M$
and let $x \in M-\{ p \}$. The sectional
curvature $K_{M}(\sigma_{x})$ of the two-plane
$\sigma_{x} \in T_{x}M$ is then called a
\textit{$p$-radial sectional curvature} of $M$ at
$x$ if $\sigma_{x}$ contains the tangent vector
to a minimal geodesic from $p$ to $x$. We denote
these curvatures by $K_{p, M}(\sigma_{x})$.
\end{definition}

In order to control the mean curvatures $H_P(x)$ of $P^{m}$ at distance $r$ from
$p$ in $N^{n}$ we introduce the following definition:

\begin{definition} The $p$-radial mean curvature function for $P$ in $N$
is defined in terms of the inner product of $H_{P}$ with the $N$-gradient of the
distance function $r(x)$ as follows:
$$
\mathcal{C}(x) = -\langle \nabla^{N}r(x),
H_{P}(x) \rangle  \quad {\textrm{for all}}\quad x
\in P \,\, .
$$
\end{definition}

In the following definition, we are going to generalize the notion of {\em radial mean convexity condition}
introduced in \cite{MP5}.

\begin{definition} (see \cite{MP5})
We say that the submanifold $P$ satisfies a {\em{radial mean convexity condition
from below}} (respectively, {\em{from above}}) from the point $p \in P$ when there exists
a radial smooth function $h(r)$, (that we call {\em a bounding function}), which satisfies
one of the following inequalities
\begin{equation}\label{eqH}
\begin{aligned}
\mathcal{C}(x)& \geq h(r(x))\,  {\textrm{for
all}}\,\,  x \in P\,\,\quad {\textrm{($h$ bounds {\it from below}})}\\
\mathcal{C}(x)& \leq h(r(x))\,  {\textrm{for
all}}\,\,  x \in P\,\,\quad {\textrm{($h$ bounds {\it from above}})}
\end{aligned}
\end{equation}
\end{definition}

The radial bounding function $h(r)$ is related with the global extrinsic
geometry of the submanifold. For example, it is obvious that  minimal submanifolds satisfy a radial mean convexity
condition from above and from below, with bounding function $h=0$. On the other hand, it can be proved,
see the works \cite{Sp}, \cite{DCW}, \cite{Pa1} and \cite{MP5},  that when the submanifold is a convex hypersurface,
then the constant function $h(r)=0$ is
a radial bounding function from below.\\

The final notion needed to describe our comparison setting is the idea of {\it radial tangency}.
If we denote by $\nabla^N r$ and $\nabla^P r$ the
 gradients of
$r$ in $N$ and $P$ respectively, then
we have the following basic relation:
\begin{equation}\label{eq2.1}
\nabla^N r = \nabla^P r +(\nabla^N r)^\bot \quad ,
\end{equation}
where $(\nabla^N r)^\bot(q)$ is perpendicular to $T_qP$ for all $q\in P$.\\

When the submanifold $P$ is totally geodesic, then $\nabla^N r=\nabla^P r$ in all points,
and, hence, $\Vert \nabla^P r\Vert =1$. On the other hand, and given the starting
point $p \in P$, from which we are measuring the distance $r$, we know that
$\nabla^N r(p)=\nabla^P r(p)$, so $\Vert \nabla^P r(p)\Vert =1$.
Therefore, the difference  $1 - \Vert \nabla^P r\Vert$ quantifies the radial {\em detour}
of the submanifold with respect the ambient manifold as seen
from the pole $p$. To control this detour locally, we apply the following

\begin{definition}

 We say that the submanifold $P$ satisfies a {\it radial tangency
 condition} at $p\in P$ when we have a smooth positive function
 $$g: P \mapsto \erre_{+} \,\, ,$$ so that
\begin{equation}
\mathcal{T}(x) \, = \, \Vert \nabla^P r(x)\Vert
\geq g(r(x)) \, > \, 0  \quad {\textrm{for all}}
\quad x \in P \,\, .
\end{equation}
\end{definition}
\begin{remark}

Of course, we always have

\begin{equation}
\mathcal{T}(x) \, =\, \Vert \nabla^P r(x)\Vert
\leq1 \quad {\textrm{for all}}
\quad x \in P \,\, .
\end{equation}
\end{remark}

\subsection{Model Spaces} \label{secModel}

As mentioned previously, the model spaces $M^m_w$ serve foremost
as com\-pa\-ri\-son controllers for the radial sectional
curvatures of $N^{n}$.

\begin{definition}[See \cite{Gri}, \cite{GreW}]
 A $w-$model $M_{w}^{m}$ is a
smooth warped product with base $B^{1} = [\,0,\, R[ \,\,\subset\,
\mathbb{R}$ (where $\, 0 < R \leq \infty$\,), fiber $F^{m-1} =
S^{m-1}_{1}$ (i.e. the unit $(m-1)-$sphere with standard metric),
and warping function $w:\, [\,0, \,R[\, \to \mathbb{R}_{+}\cup
\{0\}\,$ with $w(0) = 0$, $w'(0) = 1$, and $w(r) > 0\,$ for all
$\, r > 0\,$. The point $p_{w} = \pi^{-1}(0)$, where $\pi$ denotes
the projection onto $B^1$, is called the {\em{center point}} of
the model space. If $R = \infty$, then $p_{w}$ is a pole of
$M_{w}^{m}$.
\end{definition}

\begin{remark}\label{propSpaceForm}
The simply connected space forms $\mathbb{K}^{m}(b)$ of constant
curvature $b$ can be constructed as  $w-$models with any given
point as center point using the warping functions
\begin{equation}
w(r) = Q_{b}(r) =\begin{cases} \frac{1}{\sqrt{b}}\sin(\sqrt{b}\, r) &\text{if $b>0$}\\
\phantom{\frac{1}{\sqrt{b}}} r &\text{if $b=0$}\\
\frac{1}{\sqrt{-b}}\sinh(\sqrt{-b}\,r) &\text{if $b<0$} \quad .
\end{cases}
\end{equation}
Note that for $b > 0$ the function $Q_{b}(r)$ admits a smooth
extension to  $r = \pi/\sqrt{b}$. For $\, b \leq 0\,$ any center
point is a pole.
\end{remark}

In the papers \cite{O'N}, \cite{GreW}, \cite{Gri}, \cite{MP3} and \cite{MP4}, we have a complete description of
these model spaces, including the computation of its sectional curvature in the radial directions from the center
point, the mean curvature of the distance spheres from the center point and the volume of this spheres and the
corresponding balls.\\

In particular, in \cite{MP4} we introduced,
for any given warping function $\,w(r)\,$, the
isoperimetric quotient function $\,q_{w}(r)\,$  for the
corresponding $w-$mo\-del space $\,M_{w}^{m}\,$ as follows:
\begin{equation} \label{eqDefq}
q_{w}(r) \, = \, \frac{\Vol(B_{r}^{w})}{\Vol(S_{r}^{w})} \, = \,
\frac{\int_{0}^{r}\,w^{m-1}(t)\,dt}{w^{m-1}(r)} \quad .
\end{equation}

Then, we have the following result concerning the mean exit time
function and the torsional rigidity of a geodesic $R$-ball $B^w_R
\subseteq M^m_w$ in terms of $q_{w}$, see \cite{MP4}:

\begin{proposition}\label{propW1} Let $E_{R}^{w}$ be the solution of the Poisson Problem (\ref{eqmoments1}),
defined on the geodesic $R$-ball $B^w_R$ in the  model space $M^m_w$.

Then
\begin{equation}\label{eqEwR}
E_{R}^{w}(r) \, = \, \int_{r}^{R}\,q_{w}(t) \,dt ,
\end{equation}
and
\begin{equation}
\mathcal{A}_1(B^w_R)=\int_{B^w_R} E^w_R \, d\tilde\sigma=V_0\int_0^R
w^{m-1}(r)\Big(\int_r^R q_w(t)\, dt\Big)\, dr ,
\end{equation}
where $V_{0}$ is the volume of the unit sphere $S^{m-1}_{1}$.
Differentiating with respect to $R$ gives
\begin{equation}\label{derivative2}
\frac{d}{dR}\mathcal{A}_1(B^w_R)=q^2_w(R)\Vol(S^w_R),
\end{equation}
and an integration of the latter equality, gives us the following
alternative expression for the torsional rigidity:
\begin{equation}
\mathcal{A}_1(B^w_R)=\int_{B^w_R} q^2_w \, d\tilde\sigma.
\end{equation}
\end{proposition}
\begin{remark}\label{non-decreasing}
Since $q_w(r)>0$, it follows from (\ref{eqEwR}) that for fixed
$r$, the mean exit time function $E^w_R(r)$ is an increasing
function of $R$. Furthermore, if $q'_w(r)\geq 0$, then the average
mean exit time $\mathcal{A}_1(B^w_r)/\Vol (B^w_r)$ is also a
non-decreasing function of $r$.
\end{remark}

\subsection{The Isoperimetric Comparison space}

\label{secIsopCompSpace} Given the bounding functions $g(r)$,
$h(r)$ and the ambient curvature controller function $w(r)$ described is Subsections \ref{subsecurvature} and
\ref{secModel}, we
construct a new model space $C^{\,m}_{w, g, h}\,$, which
eventually will serve as the precise comparison space for the
isoperimetric quotients of extrinsic balls in $P$.

\begin{definition}\label{stretchingfunct}
Given a smooth positive function
$$g: P \mapsto \erre_{+} \,\, ,$$
satisfying $g(0)=1$ and $g(r(x))\leq 1\,\,{\textrm{for all \,}} x \in P$, a 'stretching' function $s$ is defined
as follows
\begin{equation}\label{eqstretching}
s(r) \, = \, \int_{0}^{r}\,\frac{1}{g(t)} \, dt \quad .
\end{equation}
It  has a well-defined inverse $r(s)$ for $s \in [\,0, s(R)\,]$
with derivative $r'(s) \, = \, g(r(s))$. In particular $r'(0)\, =
\, g(0) \, = \, 1$.
\end{definition}

\begin{definition}[\cite{MP5}] \label{defCspace}
The {\em{isoperimetric comparison space}} $C^{\,m}_{w, g, h}\,$ is
the $W-$model space with base interval $B\,= \, [\,0, s(R)\,]$ and
warping function $W(s)$ defined by
\begin{equation}
W(s) \, = \, \Lambda^{\frac{1}{m-1}}(r(s)) \quad ,
\end{equation}
where the auxiliary function $\Lambda(r)$ satisfies the following differential
equation:
\begin{equation} \label{eqLambdaDiffeq}
\begin{aligned}
\frac{d}{dr}\,\{\Lambda(r)w(r)g(r)\} \, &= \, \Lambda(r)w(r)g(r)\left(\frac{m}{g^{2}(r)}\left(\eta_{w}(r) - h(r) \right)\right) \\
&= \, m\,\frac{\Lambda(r)}{g(r)}\left(w'(r) - h(r)w(r)
\right)\quad.
\end{aligned}
\end{equation}
and the following boundary condition:
\begin{equation} \label{eqTR}
\frac{d}{dr}_{|_{r=0}}\left(\Lambda^{\frac{1}{m-1}}(r)\right) = 1
\quad .
\end{equation}

\end {definition}

We observe, that in spite of its relatively complicated
construction, $C^{\,m}_{w, g, h}\,$ is indeed a model space $M^m_W$ with a
well defined pole $p_{W}$ at $s = 0$: $W(s) \geq 0$ for all $s$
and $W(s)$ is only $0$ at $s=0$, where also, because of the
explicit construction in definition \ref{defCspace} and because of
 equation (\ref{eqTR}):  $W'(0)\, =
\, 1\,$.\\

Of course as a model space constructed in this way $M^m_W=C_{w,g,h}^{m}$ inherits all its properties
from the bounding functions $w$, $g$, and $h$ from which it is molded
in the first place. Note that, when $g(r)=1 \,\,\,{\textrm{for all \,}} r$ and $h(r)=0\,\,\,{\textrm{for all \,\,}} r$,
then
the stretching function $s(r)=r$ and $W(s(r))=w(r) \,\,\,{\textrm{for all \,}} r$, so $C_{w,g,h}^{m}$ becomes
a model space with warping function $w$, $M^m_w$.\\

 Concerning the associated volume growth
properties we note the following expressions for the isoperimetric
quotient function:

\begin{proposition} \label{propIsopFct}
Let $B_{s}^{W}(p_{W})$ denote the metric ball of radius $s$
centered at $p_{W}$ in $C^{m}_{w,g,h}$. Then the corresponding
isoperimetric quotient function is
\begin{equation} \label{eqIsopFunction}
\begin{aligned}
q_{W}(s) \, &= \, \frac{\Vol(B_{s}^{W}(p_{W}))}{\Vol(\partial
B_{s}^{W}(p_{W}))} \, \\ &= \,
\frac{\int_{0}^{s}\,W^{m-1}(t)\,dt}{W^{m-1}(s)}\,
\\ &= \,
\frac{\int_{0}^{r(s)}\,\frac{\Lambda(u)}{g(u)}\,du}{\Lambda(r(s))}
\quad .
\end{aligned}
\end{equation}
\end{proposition}

\begin{remark}\label{remarkone}
When $g(r)=1\,\,{\textrm{for all \,}} r$, the stretching function is
$s(r)=r\,\,{\textrm{for all \,}} r$, and hence
\begin{equation}
\begin{aligned}
q_{W}(s)&=q_{W}(r)\\&=\frac{\Vol(B_{r}^{W}(p_{W}))}{\Vol(\partial
B_{r}^{W}(p_{W}))}=\frac{\int_{0}^{r}\,\Lambda(u)\,du}{\Lambda(r)}
\quad .
\end{aligned}
\end{equation}

\end{remark}

 These are the spaces where the isoperimetric
bounds and the bounds on the torsional rigidity are attained. We
shall refer to the $W$-model spaces $M^{m}_{W} = C_{w,g,h}^{m}$ as the  {\em
isoperimetric comparison spaces}.

\subsection{Balance conditions}

In the paper \cite{MP4} we imposed  two further purely intrinsic
conditions on the general model spaces $M^m_w$:
\begin{definition}
A given $w-$model space $\, M^{m}_{w}\,$ is
{\em  balanced from below}  if
the following weighted isoperimetric condition is satisfied:

\begin{equation}
q_{w}(r)\,\eta_{w}(r) \, \geq 1/m \quad \text{for all} \quad r
\geq 0 \quad ,
\end{equation}
and is {\em balanced from above} if we have the inequality
\begin{equation}
 q_{w}(r)\,\eta_{w}(r) \,\leq
1/(m-1) \quad \text{for all} \quad r \geq 0 \quad .
\end{equation}

A model space is called {\em{totally balanced}} if it is balanced
both from below and from above.

The model space $M^m_w$ is easily seen to be balanced from below iff
\begin{equation}\label{diffcondition2}
\frac{d}{dr}\left(\frac{q_w(r)}{w(r)}\right)\leq 0\,\quad \text{for all} \quad r \geq
0\quad ,
\end{equation}
and balanced from above iff
\begin{equation}\label{diffcondition3}
\frac{d}{dr}\left(q_w(r)\right) \geq 0 \,\quad\text{for all} \quad r \geq 0\quad .
\end{equation}
\end{definition}

To play the comparison setting r\^{o}le in our present setting, the isoperimetric comparison spaces  must
satisfy similar types of balancing conditions:

\begin{definition} \label{defBalCond}
The model space $M_{W}^{m} \, = \, C_{w, g, h}^{m}$ is
{\em{$w-$balanced from below}} (with respect to the intermediary model space $M_{w}^{m}$) if
the following holds for all $r \in \, [\,0, R\,]$, resp. all $s \in \, [\,0, s(R)\,]$:
\begin{equation}\label{eqBalA}
q_{W}(s)\left(\eta_{w}(r(s)) - h(r(s)) \right) \, \geq g(r(s))/m \quad .
\end{equation}
\end{definition}

\begin{lemma}\label{diffcondLemma}
The model space $M_{W}^{m} \, = \, C_{w, g, h}^{m}$ is
{\em{$w-$balanced from below}} iff
\begin{equation}\label{diffcondition}
\frac{d}{dr} \left(\frac{q_W(s(r))}{g(r) w(r)}\right)\leq 0\quad .
\end{equation}
\end{lemma}
\begin{proof}
A direct differentiation using (\ref{eqIsopFunction}) but  with respect to $r$ amounts to:
\begin{equation*}
\begin{aligned}
&\frac{d}{dr}\left(\frac{q_{W}(s(r))}{g(r)w(r)}\right)\\
&=  \frac{1}{\Lambda(r)g^{3}(r)w^{2}(r)}\left(\Lambda(r)w(r)g(r) - m\left(\int_{0}^{r}\,\frac{\Lambda(t)}{g(t)}\,dt\right)\left(w'(r)
-  h(r)w(r) \right)    \right) \quad ,
\end{aligned}
\end{equation*}
which shows that inequality (\ref{diffcondition}) is equivalent to inequality
\begin{equation}\label{eqBalanceExpand2}
\Lambda(r)w(r)g(r) - m\left(\int_{0}^{r}\,\frac{\Lambda(t)}{g(t)}\,dt\right)\left(w'(r)
-  h(r)w(r) \right)\leq 0 ,
\end{equation}
which is, in turn, using (\ref{eqIsopFunction}), equivalent to inequality (\ref{eqBalA}).
\end{proof}

\begin{remark}
 In particular the $w$-balance condition from below  for $M_{W}^{m} \, = \, C_{w, g, h}^{m}$ implies that
\begin{equation} \label{eqEtaVSh}
\eta_{w}(r) \, - h(r) \, > \, 0 \quad .
\end{equation}
\end{remark}

\begin{remark}
The above definition of $w-$balance condition from below for $M_{W}^{m}$ is clearly an extension of the balance
condition from below as defined in \cite[Definition 2.12]{MP4}. The condition in that paper
is obtained precisely when $g(r) \, = \, 1$ and $h(r) \, = \, 0$
for all $r \in [\,0, R]\,$ so that $r(s) \, =\, s$, $W(s)\, = \,
w(r)$, and
\begin{equation}
q_{w}(r)\eta_{w}(r)\, \geq 1/m \quad .
\end{equation}
We observe that the differential inequality (\ref{diffcondition})
becomes (\ref{diffcondition2}) when $g(r) \, = \, 1$ and $h(r) \,
= \, 0$.
\end{remark}

As defined previously a general $w$-model space is {\em totally balanced} if it
balanced from below and from above in the sense of equations (\ref{diffcondition2}) and (\ref{diffcondition3}).
In the same way, for our present purpose, an isoperimetric comparison space $M^m_W$ can be
$w$-balanced from below in the sense of Definition \ref{defBalCond} and, moreover, considered itself as a model space,
it can be
$W-$balanced from above. In fact, these two conditions are the balancing
conditions which must be satisfied by the isoperimetric comparison
spaces in Theorems \ref{thm2.1} and \ref{thm2.2}. There are plenty of
comparison spaces $M_{W}^{m} \, = \, C_{w, g, h}^{m}$ which satisfy both conditions. Indeed, if we
differentiate equation (\ref{eqIsopFunction}) and infer the balance conditions (\ref{eqBalA})
and  $q'_W(s) \geq 0$ we get:

\begin{lemma}\label{doubleBound}
Suppose that
\begin{equation} \label{eqDoubleBound}
m(\eta_w(r(s))-h(r(s)))-g^2(r(s))\eta_w(r(s))-g(r(s))g'(r(s))>0 .
\end{equation}
Then the  isoperimetric comparison space $M_{W}^{m} \, = \,C^m_{w,g,h}$ is $w$-balanced from
below and $W-$balanced from above if and only if
\begin{equation*}
\begin{aligned}
&\frac{g(r(s))}{m(\eta_w(r(s))-h(r(s)))}\leq q_W(s)\leq\\
& \frac{g(r(s))}{m(\eta_w(r(s))-h(r(s)))-g^2(r(s))\eta_w(r(s))-g(r(s))g'(r(s))} \quad ,
\end{aligned}
\end{equation*}
which can only be satisfied by some $W(s)$ if
\begin{equation*}
 g(r)\eta_w(r) + g'(r) \, \geq \, 0 \quad{\textrm{for all \,}} r > 0 \quad .
 \end{equation*}
  On the other hand, the conditions for balance from below and balance from above
  (for standard $w-$model spaces $M^{m}_{w}$) are both open conditions on $w(r)$ in the sense
  that for these special cases where $h(r) = 0$ and $g(r) = 1$ there are several warping functions $w(r)$
  satisfying the sharp balance conditions (with strict inequalities) as well as the
  condition (\ref{eqDoubleBound}), see \cite[Observation 3.12 and Examples 3.13]{MP4}. The continuity
  of $q_W(s)$ in terms of $h(r)$, $g(r)$ and $w(r)$ then guarantees that the space of functions satisfying
  all the inequalities above is non-empty.
 \end{lemma}

\subsection{Comparison Constellations}

We now present the precise settings
where our main results take place, introducing the notion of {\em comparison constellations}.
For that purpose we shall bound the
previously introduced notions of radial curvature and tangency
by the corresponding quantities attained in some special model
spaces, called {\em isoperimetric comparison spaces} to be defined in the next subsection.

\begin{definition}\label{defConstellatNew1}
Let $N^{n}$ denote a complete Riemannian manifold with a
pole $p$ and distance function $r \, = \, r(x) \,
= \, \dist_{N}(p, x)$. Let $P^{m}$ denote an
unbounded complete and closed submanifold in
$N^{n}$. Suppose $p \in P^m$ and suppose that the following  conditions are
satisfied for all $x \in P^{m}$ with $r(x) \in
[\,0, R]\,$:
\begin{enumerate}[(a)]
\item The $p$-radial sectional curvatures of $N$ are bounded from below
by the $p_{w}$-radial sectional curvatures of of
the $w-$model space $M_{w}^{m}$:
$$
\mathcal{K}(\sigma_{x}) \, \geq \,
-\frac{w''(r(x))}{w(r(x))} \quad .
$$

\item The $p$-radial mean curvature of $P$ is bounded from below by
a smooth radial function $h(r)$, ($h$ is a {\it radial convexity function}):
$$
\mathcal{C}(x)  \geq h(r(x)) \quad.
$$

\item The submanifold $P$ satisfies a {\it radial tangency
 condition} at $p\in P$, with smooth positive function $g$ i.e. we have a smooth positive function
 $$g: P \mapsto \erre_{+} \,\, ,$$ such that
\begin{equation}
\mathcal{T}(x) \, = \, \Vert \nabla^P r(x)\Vert
\geq g(r(x)) \, > \, 0  \quad {\textrm{for all}}
\quad x \in P \,\, .
\end{equation}
\end{enumerate}
Let $C_{w,g,h}^{m}$ denote the $W$-model with the
specific warping function $W: \pi(C_{w,g,h}^{m})
\to \mathbb{R}_{+}$  constructed in
Definition \ref{defCspace}, (Subsection \ref{secIsopCompSpace}), via $w$, $g$, and $h$.
Then the triple $\{ N^{n}, P^{m}, C_{w,g,h}^{m}
\}$ is called an {\em{isoperimetric comparison
constellation bounded from below}} on the interval $[\,0, R]\,$.
\end{definition}

\begin{remark}
This definiton of {\em{isoperimetric comparison
constellation bounded from below}} was introduced in
\cite{MP5}.
\end{remark}

A \lq\lq constellation bounded from above" is given by the following dual setting,
(with respect to the definition above), considering the special $W$-model spaces $C_{w,g,h}^{m}$ with $g=1$:

\begin{definition}\label{defConstellatNew2}
Let $N^{n}$ denote a Riemannian manifold with a
pole $p$ and distance function $r \, = \, r(x) \,
= \, \dist_{N}(p, x)$. Let $P^{m}$ denote an
unbounded complete and closed submanifold in
$N^{n}$. Suppose the following  conditions are
satisfied for all $x \in P^{m}$ with $r(x) \in
[\,0, R]\,$:
\begin{enumerate}[(a)]
\item The $p$-radial sectional curvatures of $N$ are bounded from above
by the $p_{w}$-radial sectional curvatures of
the $w-$model space $M_{w}^{m}$:
$$
\mathcal{K}(\sigma_{x}) \, \leq \,
-\frac{w''(r(x))}{w(r(x))} \quad .
$$

\item The $p$-radial mean curvature of $P$ is bounded from above by
a smooth radial function $h(r)$:
$$
\mathcal{C}(x)  \leq h(r(x)) \quad.
$$
\end{enumerate}

Let $C_{w,1,h}^{m}$ denote the $W$-model with the
specific warping function $W: \pi(C_{w,1,h}^{m})
\to \mathbb{R}_{+}$ constructed, (in the same way as in Definition \ref{defConstellatNew1} above), in
Definition \ref{defCspace} via $w$, $g=1$, and $h$.
Then the triple $\{ N^{n}, P^{m}, C_{w,1,h}^{m}
\}$ is called an {\em{isoperimetric comparison
constellation bounded from above}} on the interval $[\,0, R]\,$.
\end{definition}

\begin{remark} The {\em{isoperimetric comparison
constellations bounded from above}} constitutes a generalization of the triples
$\,\{N^{n},\,P^{m},\,M^{m}_{w} \} \,$ considered in the main theorem of \cite{MP4}. This
generalization is given by the fact that we construct the isoperimetric comparison space $C_{w,g,h}^{m}$
with $g=1$, (by definition), and, when $P$ is minimal, then we consider as the bounding funtion  $h=0$. It
is straigthforward to see that, under these restrictions, $W=w$ and hence, $C_{w,1,0}^{m}=M^m_w$.
\end{remark}

\section{Isoperimetric results}
\label{secIsopRes}

We find upper bounds for the isoperimetric quotient defined as the
volume of the extrinsic sphere divided by the volume of the
extrinsic ball, in the setting given by the comparison
constellations. In order to do that, we need the following
Laplacian comparison Theorem for manifolds with a pole (see
\cite{GreW}, \cite{JK}, \cite{MP3}, \cite{MP4}, \cite{MP5} and
\cite{MM} for more details). Moreover, we shall assume along this
Section that all extrinsic balls are precompact.
\begin{theorem} \label{corLapComp} Let $N^{n}$ be a manifold with a pole $p$, let $M_{w}^{m}$ denote a $w-$model
with center $p_{w}$. Then we have the following dual Laplacian inequalities for modified distance functions:\\

(i) Suppose that every $p$-radial sectional curvature at $x \in N
- \{p\}$ is bounded  by the $p_{w}$-radial sectional curvatures in
$M_{w}^{m}$ as follows:
\begin{equation}
\mathcal{K}(\sigma(x)) \, = \, K_{p, N}(\sigma_{x})
\geq-\frac{w''(r)}{w(r)}\quad .
\end{equation}

Then we have for every smooth function $f(r)$ with $f'(r) \leq
0\,\,\textrm{for all}\,\,\, r$, (respectively $f'(r) \geq
0\,\,\textrm{for all}\,\,\, r$):
\begin{equation} \label{eqLap1}
\begin{aligned}
\Delta^{P}(f \circ r) \, \geq (\leq) \, &\left(\, f''(r) -
f'(r)\eta_{w}(r) \, \right)
 \Vert \nabla^{P} r \Vert^{2} \\ &+ mf'(r) \left(\, \eta_{w}(r) +
\langle \, \nabla^{N}r, \, H_{P}  \, \rangle  \, \right)  \quad ,
\end{aligned}
\end{equation}
where $H_{P}$ denotes the mean curvature vector
of $P$ in $N$.\\

(ii) Suppose that every $p$-radial sectional curvature at $x
\in N - \{p\}$ is bounded  by the $p_{w}$-radial sectional
curvatures in $M_{w}^{m}$ as follows:
\begin{equation}
\mathcal{K}(\sigma(x)) \, = \, K_{p, N}(\sigma_{x})
\leq-\frac{w''(r)}{w(r)}\quad .
\end{equation}

Then we have for every smooth function $f(r)$ with $f'(r) \leq
0\,\,\textrm{for all}\,\,\, r$, (respectively $f'(r) \geq
0\,\,\textrm{for all}\,\,\, r$):
\begin{equation} \label{eqLap2}
\begin{aligned}
\Delta^{P}(f \circ r) \, \leq (\geq) \, &\left(\, f''(r) -
f'(r)\eta_{w}(r) \, \right)
 \Vert \nabla^{P} r \Vert^{2} \\ &+ mf'(r) \left(\, \eta_{w}(r) +
\langle \, \nabla^{N}r, \, H_{P}  \, \rangle  \, \right)  \quad ,
\end{aligned}
\end{equation}
where $H_{P}$ denotes the mean curvature vector of $P$ in $N$.
\end{theorem}

The isoperimetric inequality (\ref{eqIsopGeneralA}) below has been stated and proved previously
in \cite[Theorem 7.1]{MP5}. On the other hand, the isoperimetric inequality (\ref{eqIsopGeneralB}) has been
stated and proved in \cite{MP4}, but only under the assumption that $P$ is minimal and that the model space
satisfies a more restrictive balance condition, see Remark f. For completeness we therefore give a sketch of the
proof of inequality (\ref{eqIsopGeneralB}) below.

\begin{theorem} \label{thmIsopGeneral1} There are two dual settings to be considered:\\
(i) Consider an isoperimetric comparison
constellation bounded from below $\{ N^{n}, P^{m}, C_{w,g,h}^{m}
\}$.
Assume that the isoperimetric comparison space
 $\, C_{w,g,h}^{m}\,$ is $w$-balanced from below.
Then
\begin{equation} \label{eqIsopGeneralA}
\frac{\Vol(\partial D_{R})}{\Vol(D_{R})} \leq
\frac{\Vol(\partial
B^{W}_{s(R)})}{\Vol(B^{W}_{s(R)})} \, \leq \,
\frac{m}{g(R)}\left(\eta_{w}(R) -
h(R)\right)\quad .
\end{equation}
where $s(R)$ is the {\em stretched} radius given by Definition \ref{stretchingfunct}.\\

\noindent (ii) Consider an isoperimetric comparison
constellation bounded from above $\{ N^{n}, P^{m}, C_{w,1,h}^{m}
\}$.
 Assume that the isoperimetric comparison space
 $\, C_{w,1,h}^{m}\,$ is $w$-balanced from below.
Then
\begin{equation} \label{eqIsopGeneralB}
\frac{\Vol(\partial D_{R})}{\Vol(D_{R})} \geq
\frac{\Vol(\partial
B^{W}_{R})}{\Vol(B^{W}_{R})} .
\end{equation}

If equality holds in (\ref{eqIsopGeneralB}) for some fixed radius $R >0$, then $D_R$ is a cone
in the ambient space $N^n$.
\end{theorem}

\begin{proof} The proof starts from the same point for both inequalities.
As in \cite{MP5}, we define a second order differential operator
$\LL$ on functions $f$ of one real variable as follows:
\begin{equation} \label{eq3.7}
\LL f(r) \, = \, f''(r)\,g^{2}(r) +
f'(r)\left((m-g^{2}(r))\,\eta_{w}(r) - m\,h(r)
\right) \quad ,
\end{equation}
and consider the smooth solution $\psi(r)$ to the
following Dirichlet--Poisson problem:
\begin{equation}   \label{eqPoisson1}
\begin{aligned}
\LL \psi(r) &= -1\,\,\,\,\text{on}\,\,\,\, [\,0, R]\quad ,\\
\psi(R) &=0\quad .
\end{aligned}
\end{equation}
The ODE is equivalent to the following:
\begin{equation}\label{eqPoisson2}
\psi''(r) +
\psi'(r)\left(-\eta_{w}(r) + \frac{m}{g^{2}(r)}\left(\eta_{w}(r) - \,h(r)\right)\right)\, = \, -\frac{1}{g^{2}(r)}\quad .
\end{equation}
The solution is constructed via the auxiliary function
$\Lambda(r)$ from equation (\ref{eqLambdaDiffeq}) and it is given, as it can be
seen in \cite{MP5}, by:

\begin{equation}\label{derivative}
\begin{aligned}
\psi'(r) \, &= \, \Gamma (r)\, = \,
\frac{-1}{g(r)\,\Lambda(r)}\,\int_{0}^{r}\,\frac{\Lambda(t)}{g(t)}\,dt
\, \\ &= - \frac{\Vol(B^{W}_{s(r)})}{g(r)\Vol(\partial
B^{W}_{s(r)})}\, = - \frac{q_{W}(s(r))}{g(r)}\quad ,
\end{aligned}
\end{equation}
and then
\begin{equation} \label{eqMeanExitMonster}
\begin{aligned}
\psi(r) \,&= \,
\int_{r}^{R}\frac{1}{g(u)\,\Lambda(u)}\,\left(\int_{0}^{u}\,\frac{\Lambda(t)}{g(t)}\,dt\,\right)\,du
\\ &= \int_{r}^{R} \frac{q_{W}(s(u))}{g(u)} \, du
\,= \int_{s(r)}^{s(R)} \,q_{W}(t) \, dt \quad .
\end{aligned}
\end{equation}

We must recall, as it was pointed out in Remark \ref {remarkone}, that, when we consider a comparison
constellation bounded from above, as in the statement (ii) of the Theorem, then $g(r) =1$ in
(\ref{eq3.7}) and (\ref{eqPoisson2}), so $s(r)=r$, and
\begin{equation*}
\psi'(r)=-q_W(r)=-\frac{\Vol(B^W_r)}{\Vol(S^W_r)} .
\end{equation*}

Then - because of the balance condition (\ref{eqBalA}) and equation (\ref{eqPoisson2})
- the function $\psi(r)$ enjoys the following inequality:
\begin{equation} \label{eqYellow}
\psi''(r) - \psi'(r)\,\eta_{w}(r) \, \geq \, 0 \
\quad .
\end{equation}

The second common step to prove isoperimetric inequalities (\ref{eqIsopGeneralA}) and (\ref{eqIsopGeneralB}),
is to transplant $\psi(r)$ to
$D_R$ defining
$$\psi: D_R \longrightarrow \erre; \quad \psi(x):=\psi(r(x))\, .$$

Now, we are going to focus attention on the isoperimetric inequality (\ref{eqIsopGeneralB}).
In this case, we have that the sectional
curvatures of the ambient manifold are bounded from above, inequality (\ref{eqYellow}), that the p-radial mean
curvature of $P$ is bounded from above by $h(r)$, and that $\eta_w(r)-h(r) >0 \quad {\textrm{for all \,}} r>0$. Then,
applying now the Laplace inequality (\ref{eqLap2}) in Theorem \ref{corLapComp}
for the transplanted function $\psi(r)$  we have the following
comparison,
\begin{equation}\label{lemaJK2}
\begin{aligned}
\Delta^{P}\psi(r(x)) \, &\leq \,
\left(\psi''(r(x))- \psi'(r(x))\eta_{w}(r(x))\right)\Vert\nabla^P r\Vert^2 \\
& \phantom{mm}+ m\psi'(r(x))\left( \eta_{w}(r(x)) -
h(r(x))\right)\, \\ &\leq \, \LL \psi(r(x)) \, = \, -1 \, = \,
\Delta^{P}E(x) \quad .
\end{aligned}
\end{equation}

Applying the divergence theorem, using the unit normal $\nabla^P r/\,\Vert \nabla^P r\Vert$ to $\partial D_r$,
we get, as in \cite{Pa1}, but now for submanifolds with $p$-radial mean curvature bounded from above by $h(r)$:

 \begin{equation}\label{eqVolFrac2}
    \begin{aligned} \Vol(D_{R})&  \leq
    \int_{D_{R}} -\Delta^{P} \psi(r(x))\, d\sigma \\ &=
    -\Gamma(R)\int_{\partial D_{R}}\Vert\nabla^P r\Vert\, d\sigma \\&  \leq
    -\Gamma(R)\Vol(\partial D_{R}) \quad .
    \end{aligned}
    \end{equation}
    which shows the isoperimetric inequality (\ref{eqIsopGeneralB}), because in this case, and in view of
    remark \ref{remarkone}, we have that
    $$\Gamma(r)=\psi'(r)=-q_W(r)=-\frac{\Vol(B^W_r)}{\Vol(S^W_r)} .$$

    To prove the equality assertion, we note that equality in (\ref{eqIsopGeneralB}) for some fixed $R>0$ implies
    that the inequalities in (\ref{lemaJK2}) and (\ref{eqVolFrac2}) become equalities. Hence,
    $\Vert\nabla^P r\Vert=1=\Vert\nabla^N r\Vert$ in $D_R$, so $\nabla^P r =\nabla^N r$ in $D_R$. Then, all the
    geodesics in $N$ starting at $p$ thus lie in $P$, so $D_R=\exp_p(\widetilde D_R)$, with $\widetilde D_R$ being
    the $0$-centered $R$-ball in $T_pP$. Therefore, $D_R$ is a cone in $N$.\\

 Inequality (\ref{eqIsopGeneralA}) is proved in the same way, see \cite{MP5}, but using the Laplace
 inequality (\ref{eqLap1}) to the transplanted function $\psi(r)$. In this case, we are assuming that the sectional
 curvatures of the ambient manifold are bounded from below and the $p$-radial mean curvature of the
 submanifold is bounded from below  by the function $h(r)$. Under these conditions, we have
   \begin{equation}\label{lemaJK3}
\Delta^{P}\psi(r(x)) \, \geq \, \LL \psi(r(x)) \, = \, -1 \, = \, \Delta^{P}E(x) \quad .
\end{equation}

Then, we obtain the result applying the divergence theorem as
before and taking into account that in this case the derivative of
$\psi(r)$ is
$$\Gamma(R)=\psi'(R)=- \frac{\Vol(B^{W}_{s(R)})}{g(R)\Vol(\partial
B^{W}_{s(R)})} .$$

\end{proof}

A corollary of the proof of Theorem \ref{thmIsopGeneral1} is the
following

\begin{proposition} \label{relphiE}Let us consider the isoperimetric model space $M^m_W=C^m_{w,g,h}$. Then
$$\psi(r)=E^W_{s(R)}(s(r))\,\,{\textrm{for all \,}} r \in [0,R]\quad ,$$
where $s$ is the stretching function defined in equation
(\ref{eqstretching}) and
$$E^W_{s(R)}:B^W_{s(R)} \longrightarrow \erre\quad ,$$
is the solution of the Poisson problem

\begin{equation}   \label{eqPoisson3}
\begin{aligned}
\Delta^{C_{w,g,h}^{m}} E(s) &= -1\,\,\,\,\text{on}\,\,\,\, B^{W}_{s(R)}\quad ,\\
E &=0\,\,\,\,\text{on}\,\,\,\, \partial B^{W}_{s(R)}\quad.
\end{aligned}
\end{equation}
\end{proposition}
\begin{proof}

This follows directly from Proposition \ref{propW1} by applying (\ref{eqMeanExitMonster}).

\end{proof}

The proof of the next Corollary \ref{corVolBound}, (where we assume that the submanifold $P$ has bounded $p$-radial
mean curvature from above or from below), follows the same formal steps as the corresponding results for minimal
submanifolds, which can be founded in \cite{MP5}, \cite{Pa2}, and in \cite{MP4}. As in these proofs, the co-area
formula, see \cite{Cha1}, plays here a fundamental r\^{o}le.

\begin{corollary} \label{corVolBound} Again we consider the two dual settings:\\
(i) Let $\{ N^{n}, P^{m}, C_{w,g,h}^{m} \}$ be a
comparison constellation bounded from below on the interval $[\,0,
R]\,$, as in statement (i) of Theorem \ref{thmIsopGeneral1}.

Then
\begin{equation}\label{CompVolUp}
\Vol(D_{r}) \, \leq \, \Vol(B_{s(r)}^{W})
\,\,\,\,\,\,\textrm{for every} \,\,\, r \in [\,0,
R] \,\, .
\end{equation}

(ii) Let $\{ N^{n}, P^{m}, C_{w,1,h}^{m} \}$ be a
comparison constellation bounded from above on the interval $[\,0,
R]\,$, as in statement (ii) of Theorem \ref{thmIsopGeneral1}.

Then
\begin{equation}\label{CompVolDown}
\Vol(D_{r}) \, \geq \, \Vol(B_{r}^{W})
\,\,\,\,\,\,\textrm{for every} \,\,\, r \in [\,0,
R] \,\, .
\end{equation}

Equality in (\ref{CompVolDown}), for  all $r\in[0,R]$ and some
fixed radius $R>0$ implies that $D_R$ is a cone in $N^n$, using
the same arguments as in the proof of Theorem
\ref{thmIsopGeneral1}.

\end{corollary}

\section{Symmetrization into model spaces}\label{SecSymm}

As in \cite{MP4} we use  the concept of Schwarz--symmetrization as considered in e.g.  \cite{Ba}, \cite{Po}, or,
more recently, in \cite{Mc} and \cite{Cha2}. We review some facts
about this instrumental tool.
\begin{definition}\label{Symm}
Suppose $\,D\,$ is a precompact open connected domain in
$\,P^m\,$. Then the $w-$model space symmetrization of $\,D\,$ is
denoted by $D^{*}$ and is defined to be the unique
$p_{w}-$centered ball $\,D^{*}\, = \, B^{w}(D)\,$ in $\,M^m_w\,$
satisfying $\,\Vol(D)\,=\,\Vol(B^{w}(D))\,$. In the particular
case where $D$ is actually an extrinsic metric ball $D_{R}$ in
$\,P\,$ of radius $R$ we may write
$$
D_{R}^{*} \, = \,B^{w}(D) \,= \, B^{w}_{T(R)} \quad ,
$$
where $\,T(R)\,$ is some increasing function of $R$ which depends
on the geometry of $\,P\,$, according to the defining property:
$$
\,\Vol(D_{R})\,=\,\Vol(B^{w}_{T(R)})\, \quad .
$$
\end{definition}

We also introduce the notion of a symmetrized function on the
symmetrization $\,D^{*}\,$ of $\,D\,$ as follows.

\begin{definition}\label{defsimfunc}
Let $f$ denote a nonnegative function on $\,D\,$
$$
f: D \subseteq P \rightarrow \mathbb{R}^+ \cup
    \{0\}\quad .
$$
   For $t >0$ we let
$$
\,D(t)\,=\, \{x \in D \,|\, f(x) \geq t\}\, \quad .
$$
    Then the symmetrization of $f$ is
    the function $f^{\,*}: D^{*} \rightarrow \mathbb{R} \cup
    \{0\}\,$
    de\-fi\-ned by
$$
f^{\,*}(x^{*})= \sup\{t\,|\, x^{*} \in D(t)^{*}\,\}\quad .
$$
\end{definition}

\begin{proposition}
The symmetrized objects $f^{\,*}$ and $D^{*}$ satisfy the
following properties:
\begin{enumerate}
    \item The function $f^{\,*}$ depends only on the geodesic distance
    to the center $\,p_{w}\,$ of the ball $\,D^{*}\,$ in $\,M_{w}^{m}\,$ and is non-increasing.

    \item The functions $f$ and $f^{\,*}$ are equimeasurable in the
    sense that
\begin{equation}\label{eqVol}
\Vol_P(\{x \in D \,|\, f(x) \geq t\})=
    \Vol_{M^m_{w}}(\{x^{*} \in D^{*} \,|\,  f^{\,*}(x^{*}) \geq t\})
\end{equation}
for all $\,t \geq 0\,$.
    In particular, for all $t >0$, we have
\begin{equation}\label{eqMass}
     \int _{D(t)} f \,d\sigma \, \leq \, \int_{D(t)^{*}} f^{\,*} \,d\tilde\sigma \quad .
 \end{equation}
\end{enumerate}
\end{proposition}

\begin{remark}
 The proof of these properties follows the proof of the
classical Schwarz symmetrization using the 'slicing' technique for
symmetrized volume integrations and comparison -- see e.g.
\cite{Cha2}.
\end{remark}
\medskip
In the proof of both Theorem \ref{thm2.1} and Theorem \ref{thm2.2} in Section \ref{MainSect},  we shall consider
a symmetric model space rearrangement of the extrinsic ball $D_R$ as it has been described in Definition \ref{Symm}
and Definition \ref{defsimfunc} \, , namely, a symmetrization of $D_R$ which is a  geodesic $T(R)$-ball in the model
space $M^m_W$ such that $\text{vol}(D_R)=\text{vol}(B^W_{T(R)})$, together the symmetrization of the transplanted
radial function $\psi: D_R \longrightarrow \erre$ of the solution  of the Poisson problem (\ref{eqPoisson1}) in $[0,R]$.
We know (see Proposition \ref{relphiE}) that $\psi(r)=E^W_{s(R)}(s(r))$, where $E^W_{s(R)}$ is the solution of
the Poisson problem (\ref{eqPoisson3}).\\

This symmetrization is a function $\psi^*: B^W_{T(R)}\longrightarrow \erre$  which satisfies the property that
inequality (\ref{eqMass}) becomes an equality. This property becomes a crucial fact in the proof of
Theorems \ref{thm2.1} and \ref{thm2.2}.

\begin{theorem}\label{eqschwarz}
Let $\psi^*: B^W_{T(R)}\longrightarrow \erre$ be the symmetrization of the transplanted radial
function $\psi: D_R \longrightarrow \erre$ of the solution  of the Poisson problem (\ref{eqPoisson1}) in $[0,R]$. Then
    \begin{equation}
    \int_{D_R} \psi d\sigma=\int_{B^W_{T(R)}}\psi^*
    d\tilde\sigma\quad .
    \end{equation}
    \end{theorem}
    \begin{proof}

First of all, we are going to define $\psi^*$. To do that, let us consider  $T= \max_{[0,R]} \psi$.
    On the other hand, and given $t \in [0,T]$, let us define
    the sets
    $$
    D(t)=\{x \in D_R \,|\, \psi(r(x)) \geq t\}\quad ,
    $$
    and
    $$
    \Gamma(t)=\{x \in D_R \,|\, \psi(r(x))= t\} \quad .
    $$
    As $\psi(r(x))=E^W_{s(R)}(s(r(x)))\,\,\,{\textrm{for all \,}} x \in D_R$,  then $\psi$ is
    radial and non-increasing, its maximum $T$ will be
    attained at $r=0$, $D(t)$ is the extrinsic ball in $P$ with radius
    $a(t):=\psi^{-1}(t)$, (we denote it as $D_{a(t)}$),
    and $\Gamma(t)$ is its boundary, the extrinsic sphere with
    radius $a(t)$, $\partial D_{a(t)}$. We have too that
    $D(0)=D_R$ and $D(T)=\{p\}$, the center of the
    extrinsic ball $D_R$.

    We consider the symmetrizations of the sets $D(t) \subseteq P$,
    namely, the geodesic balls
    $D(t)^*=B^W_{\tilde r(t)}$ in $M^m_W$
    such that
     $$\Vol(D(t))=\Vol(D_{a(t)})=\Vol(B^W_{\tilde r(t)})\quad .$$

Hence, we have defined a non-increasing function
    $$\tilde r: [0,T] \longrightarrow [0,T(R)]; \,\,\,\tilde r=\tilde r(t)\quad ,$$
    defined as the radius $\tilde r(t)$ from the center $\tilde p$ of the model space $C_{w,g,h}^{m}$ such
    that $\Vol(B^W_{\tilde r(t)})=\Vol(D(t))=\Vol(D_{a(t)})$, (and hence, $\tilde r(0)=T(R)$ and $\tilde r(T)=0$),
    with inverse $$\phi:[0, T(R)] \longrightarrow [0,T];\,\,\, \phi= \phi(\tilde r)\quad ,$$
    such that $\phi'(\tilde r(t))=\frac{1}{\tilde r'(t)}\,\,\,{\textrm{for all \,}} t\in [0,T]$.

    Thus, given $\tilde x \in B^W_{T(R)}$, and taking into account
    that $$B^W_{T(R)} =\cup_{t \in [0,T]} \partial D(t)^*=\cup_{t \in [0,T]}S^W_{\tilde r(t)}\,\, ,$$ there exists
    some biggest value $t_0$
    such that $r_{\tilde p}(\tilde x)=\tilde r(t_0)$, (and hence, $\tilde x \in D(t_0)^*$).
Therefore, in accordance with Definition \ref{defsimfunc},
    the symmetrization of $\psi: D_R \longrightarrow \erre$  is a function $\psi^*: B^W_{T(R)}\longrightarrow \erre$
    defined as
    \begin{equation}\label{eqPsi}
    \psi^*(\tilde x)=E^{W*}_{s(R)}(s(r_{\tilde p}(\tilde x))= t_0=\phi(\tilde r(t_0)) .
    \end{equation}

\begin{remark} We pause to make two observations:\\
(i) Note that $\psi^*$ is a radial function, $\psi^*(\tilde x)=\psi^*(\tilde r(\tilde x))=\psi^*(\tilde r)$.
Therefore, for all $\tilde r \in [0,T(R)]$ and $t \in [0,T]$, we have
    \begin{equation}\label{eq1} \psi^{*'}(\tilde r)=\phi'(\tilde r(t))=\frac{1}{\tilde r'(t)}\,\, .\end{equation}\\
(ii) Let  $T(R)$ be the radius such that $\Vol(B^W_{T(R)})=\Vol(D_R)$, and let $s(R)$ be the \lq\lq stretched"
radius $s(R)=\int_0^R \frac{1}{g(t)} dt$.

    As the comparison constellation is bounded from below, and by virtue of  inequality (\ref{CompVolUp}) in
    Corollary \ref{corVolBound}, we have, for all $t \in [0,T]$, $\Vol(B^W_{\tilde r(t)}) =\Vol(D_{a(t)}) \leq \Vol(B^W_{s(a(t))})$,
    so $\tilde r(t) \leq s(a(t)) \,\,{\textrm{for all \,}} t \in [0,T]$ and then
    \begin{equation}\label{size1} T(R)=b(0)\leq s(a(0))=s(R)\quad .
    \end{equation}
    \end{remark}
    \vspace{0.4cm}

     By definition of $\psi^*$, we have  $\psi^* =\phi\circ \tilde r$ on $B^W_{T(R)}$, see (\ref{eqPsi}). Then,
     using the formula for integration in a disc in a model space (\cite[p. 47]{Cha1}) we get
    \begin{equation}
    \begin{aligned}
    \int_{B^W_{T(R)}}\psi^* d\tilde\sigma&=\int_{B^W_{T(R)}}\phi\circ \tilde r d\tilde\sigma
    \\&=\int_{S^{0,m-1}_1}d A(\xi)\{\int_0^{T(R)}\phi(\tilde r) W^{m-1}(\tilde r) d\tilde r\}\\
    &=\int_0^{T(R)}\phi(\tilde r) \Vol(S^{0,m-1}_1) W^{m-1}(\tilde r) d\tilde r\\&= \int_0^{T(R)}\phi(\tilde r) \Vol(S^{W}_{\tilde r})
    d\tilde r \quad .
    \end{aligned}
    \end{equation}
     Now, we change the variable using the bijective, (monotone decreasing), function
     $\tilde r:[0,T] \longrightarrow [0,T(R)];\,\, \tilde r(0)=T(R),\,\, \tilde r(T)=0$, so
     \begin{equation}
     \int_0^{T(R)}\phi(\tilde r) \Vol(S^{W}_{\tilde r}) d\tilde r=\int_T^0\phi(\tilde r(t)) \Vol(S^{W}_{\tilde r(t)})
     \tilde r'(t)dt \quad .
     \end{equation}

     But we know that $\phi(\tilde r(t))=t \,\,\,{\textrm{for all \,}} t \in [0,T]$, and, on the other hand,
     denoting as $V(t)=\Vol(B^{W}_{\tilde r(t)})=\Vol(D(t))\,\,\,{\textrm{for all \,}} t\in [0,T]$, we have
     \begin{equation}\label{eq2}
     V'(t)=\Vol(S^{W}_{\tilde r(t)})\tilde r'(t) \,\,{\textrm{for all \,}} t\in [0,T]\quad ,
     \end{equation}

     and hence
     \begin{equation}
     \int_T^0\phi(\tilde r(t)) \Vol(S^{W}_{\tilde r(t)}) \tilde r'(t)dt=-\int_0^T t
     V'(t)dt\quad .
     \end{equation}

     Now, we apply co-area formula to the following setting: we have the transplanted function
     $\psi:D_R \longrightarrow \erre$, and the sets $D(t)$, with their boundaries $\Gamma(t)$. We have, by definition,
     that $V(t)=\Vol(D(t))$, so
     \begin{equation}\label{eq3}
     V'(t)=-\int_{\Gamma(t)}\Vert \nabla^P \psi\Vert^{-1}
     d\sigma_t\quad .
     \end{equation}

     Hence, putting together all the equalities before, taking into account that
     $\psi\vert_{\Gamma(t)}=t\,\,\,{\textrm{for all \,}} t \in [0,T]$, and using the co-area formula
     again (\cite[equation (4) in Theorem 1, p. 86]{Cha1}), we conclude
     \begin{equation}
     \begin{aligned}
     \int_{B^W_{T(R)}}&\psi^* d\tilde\sigma=-\int_0^T t V'(t)dt
     =\int_{D_R}\psi
     d\sigma \quad .
     \end{aligned}
     \end{equation}
    \end{proof}

\section{Main results}\label{MainSect}

By definition, the torsional rigidity $\mathcal{A}_1(D_R)$ is the
$D_R$-integral of the mean exit time function $E_{R}(x)$ from $x$
in $D_{R}$. We note that for most minimally immersed submanifolds
$\,P^{m}\,$ in the flat Euclidean spaces $\,\mathbb{R}^{n}\,$ with
the obvious choice of comparison model space, $\,M^{m}_{W}\,=
\,\mathbb{R}^{m}\,$, $\,W(r)\, = \, r\,$, we have (see \cite{Ma1},
\cite{Pa2}):
$$
E_{R}(x) \, = \, E_{R}^{W}(r(x)) \quad \textrm{for all} \quad x
\in D_{R} \quad,
$$
$$
\textrm{but also} \quad \Vol(D_{R}) \,>\, \Vol(B_{R}^{W}) \quad,
$$
$$
\textrm{so that} \quad \tors(D_{R}) \, > \, \tors(B_{R}^{W}) \quad
.
$$

In this sense Theorem \ref{thm2.1} is a generalization of this
fact, when we assume that the ambient space has sectional
curvatures bounded from below, and that the mean curvature of the
submanifold is controlled along the radial directions from the
pole. These assumptions includes minimality and convexity of the
submanifold. This result is based on previous geometrical and
analytical considerations from \cite{MP5}.

\begin{theorem} \label{thm2.1}
Let $\{ N^{n}, P^{m}, C_{w,g,h}^{m} \}$ denote a comparison
constellation boun\-ded from below in the sense of Definition \ref{defConstellatNew1}. Assume
that $M^m_W=C_{w,g,h}^{m}$ is
$w$-balanced from below, and $W-$balanced from above. Let $D_R$ be a precompact extrinsic $R$-ball in $P^m$,
 with center at a point $p \in P$ which also serves as a pole in $N$.  Then
\begin{equation} \label{eqMain}
\mathcal{A}_1(D_R) \geq \mathcal{A}_1\left(B^{W}_{T(R)}\right) \quad ,
\end{equation}
where $B^{W}_{T(R)}$ is the Schwarz symmetrization of $D_R$ in the
$W$-model space $C_{w,g,h}^{m}$, i.e., it is the geodesic ball in
$C_{w,g,h}^{m}$ such that $\Vol(D_R)= \Vol(B^{W}_{T(R)})$.
\end{theorem}

\begin{proof}[Proof of Theorem \ref{thm2.1}]

 Given the solution $E_R$ to the  Dirichlet-Poisson
equation on $D_R$, we compare it with the transplanted function
$\psi(r(x))$, defined on $D_R$ as the radial solution of equation (\ref{eqPoisson1}) in the proof of Theorem
\ref{thmIsopGeneral1}. In fact, by inequality (\ref{lemaJK3}) we
have that $\psi-E_R$ is a subharmonic function with
$E_R(R)=\psi(R)=0$, so, applying Maximun Principle, $$E_R \geq
\psi \,\,\text{on}\,\, D_R\quad . $$

Using this inequality and Proposition \ref{eqschwarz}, we have
\begin{equation}\label{eq4}
\mathcal{A}_1(D_R) =\int_{D_R} E_R d\sigma \geq \int_{D_R} \psi
d\sigma \quad =\int_{B^W_{T(R)}}\psi^* d\tilde\sigma .
\end{equation}
\bigskip
The symmetrized function $\psi^*: B^W_{T(R)} \longrightarrow \erre$ is a radial function, but it does not
necessarily satisfy the Poisson equation on $B^W_{T(R)}$. Then, we are going to compare $\psi^*$ with the radial
solution $E^W_{T(R)} : B^W_{T(R)} \longrightarrow \erre$ of the Dirichlet-Poisson problem

    \begin{equation}
\begin{aligned}
\Delta^{C^m_{w,g,h}}E&= -1\,\,\,\,\text{on}\,\,\,\, B^W_{T(R)}\\
E &=0\,\,\,\,\text{on}\,\,\,\, \partial B^W_{T(R)} .
\end{aligned}
\end{equation}

 To do that, we shall prove the following (the proof is given below after finishing the proof of Theorem \ref{thm2.1})
 \begin{proposition}\label{lemineqphiE}
 \begin{equation}\label{ineqphiE}
 \psi^{*'}(\tilde r) \leq E^{W'}_{T(R)}(\tilde r) \,\,\,{\textrm{for all \,}} \tilde r \in [0,T(R)] .
 \end{equation}
 \end{proposition}
\vspace{0.4cm}
 Assuming (\ref{ineqphiE}) for a moment, integrating from $\tilde r$ to $T(R)$ both sides of inequality
 (\ref{ineqphiE}), and taking into account that
 $$\psi^*(T(R))=\phi(T(R))=0=E^W_{T(R)}(T(R))\quad ,$$
 we obtain, for all $\tilde r \in [0,T(R)]$,
 \begin{equation}
 -\psi^*(\tilde r)=\int_{\tilde r}^{T(R)}\psi^{*'}(l)dl \leq \int_{\tilde r}^{T(R)}E^{W'}_{T(R)}(l)  dl
 =-E^W_{T(R)}(\tilde r)\quad ,
\end{equation}
 and hence,
$$\psi^*(\tilde r) \geq E^W_{T(R)}(\tilde r) \,\,\,\,{\textrm{for all \,}} \tilde r \in [0,T(R)]\quad .$$
Therefore,

\begin{equation}
   \begin{aligned}
   \mathcal{A}_1(D_R) &=\int_{D_R} E_R d\sigma \geq \int_{D_R} \psi d\sigma
   =\int_{B^W_{T(R)}}\psi^* d\tilde\sigma\\ &\geq \int_{B^W_{T(R)}}E^W_{T(R)}d\tilde\sigma
   = \mathcal{A}_1(B^W_{T(R)})\quad ,
   \end{aligned}
   \end{equation}
   and the Theorem is proved.
\end{proof}

\begin{proof}[Proof of Proposition \ref{lemineqphiE}]
Using equations (\ref{eq1}), (\ref{eq2}) and (\ref{eq3}), we have
that
\begin{equation}
\begin{aligned}
\psi^{*'}(\tilde r)&=\frac{1}{\tilde r'(t)}=-\frac{\Vol(S^W_{\tilde r(t)})}{\int_{\Gamma(t)}\Vert
\nabla^P \psi\Vert^{-1} d\sigma_t}\quad .
\end{aligned}
\end{equation}

As $\psi(r)$ is radial, we have
\begin{equation}
\Vert\nabla^P \psi(r)\Vert=\vert \psi'(r)\vert\Vert \nabla^P
r\Vert\geq \vert \psi'(r)\vert g(r)\quad ,
\end{equation}
 so, as $\Gamma(t)=\partial D_{a(t)}$ for all $t \in [0,T]$, we
 have that
 \begin{equation}
 \begin{aligned}
  \int_{\Gamma(t)}\Vert \nabla^P \psi\Vert^{-1} d\sigma_t&=\frac{1}{\vert \psi'(a(t))\vert}\int_{\partial D_{a(t)}}
  \Vert \nabla^P r\Vert^{-1}\\ &\leq \frac{1}{\vert \psi'(a(t))\vert g(a(t))}{\Vol(\partial
  D_{a(t)})}\quad ,
  \end{aligned}
  \end{equation}
  and hence, by equation (\ref{derivative})
  \begin{equation}
   \begin{aligned}
 \psi^{*'}(\tilde r(t)) &\leq -\vert \psi'(a(t))\vert g(a(t))\frac{\Vol(S^W_{\tilde r(t)})}{\Vol(\partial
 D_{a(t)})}\\
 &=-\frac{\Vol(B^W_{s(a(t))})}{\Vol(S^W_{s(a(t))})}\frac{\Vol(S^W_{\tilde r(t)})}{\Vol(\partial
 D_{a(t)})}\quad .
 \end{aligned}
 \end{equation}

But we have that (see Remark j in the proof of Theorem \ref{eqschwarz} and inequality (\ref{CompVolUp}) in
Corollary \ref{corVolBound})
$$\tilde r(t) \leq s(a(t))\,\,{\textrm{for all \,}} t \quad ,$$
so, since $q_W'(r) \geq 0\,$, we get:
\begin{equation}
\frac{\Vol(B^W_{\tilde r(t)})}{\Vol(S^W_{\tilde r(t)})}\leq
\frac{\Vol(B^W_{s(a(t))})}{\Vol(S^W_{s(a(t))})}\quad .
\end{equation}
Therefore, as $\Vol(B^W_{\tilde r(t)})=\Vol(D_{a(t)})$,
\begin{equation}
\psi^{*'}(\tilde r(t)) \leq-\frac{\Vol(D_{a(t)})}{\Vol(\partial
D_{a(t)})}\quad .
\end{equation}

Now, we apply again the isoperimetric inequality of Theorem
\ref{thmIsopGeneral1} (i), the fact that $\tilde r(t) \leq s(a(t))$ and
that  $q_W'(r) \geq 0$ to obtain finally
\begin{equation}
\begin{aligned}
\psi^{*'}(\tilde r(t))& \leq-\frac{\Vol(D_{a(t)})}{\Vol(\partial
D_{a(t)})}\leq -\frac{\Vol(B^W_{s(a(t))})}{\Vol(S^W_{s(a(t))})}
\\ &\leq -\frac{\Vol(B^W_{\tilde r(t)})}{\Vol(S^W_{\tilde r(t)})}
\, =E^{W'}_{T(R)}(\tilde r(t))\quad .
\end{aligned}
\end{equation}

\end{proof}
Theorem \ref{thm2.2} below is a generalization of the result \cite[Theorem 2.1]{MP4}. In that paper, we
obtained an upper bound for
the torsional rigidity of the extrinsic domains of a {\em{minimal}}
submanifold.  We assume now that the radial mean curvature of the submanifold is bounded from above,
and, as in \cite{MP4}, that the ambient manifold has sectional curvatures bounded from above.
Hence, we have the following generalization to
submanifolds which are not necessarily minimal:

\begin{theorem} \label{thm2.2}
Let $\{ N^{n}, P^{m}, C_{w,1,h}^{m} \}$  denote a comparison
constellation boun\-ded from above. Assume that
$M^m_W=C_{w,1,h}^{m}$ is $w$-balanced from below, $W-$ba\-lan\-ced
from above, and that it has infinite volume. Let $D_R$ be a
precompact extrinsic $R$-ball in $P^m$,
 with center at a point $p \in P$ which also serves as a pole in $N$.  Then
\begin{equation} \label{eqMain3}
\mathcal{A}_1(D_R) \leq \mathcal{A}_1(B^{W}_{T(R)}) \quad ,
\end{equation}
where $B^{W}_{T(R)}$ is the Schwarz symmetrization of $D_R$ in the
$W$-model space $M^m_W$, i.e., it is the geodesic ball in $M^m_W$
such that $\Vol(D_R)= \Vol(B^{W}_{T(R)})$.
Equality in
(\ref{eqMain3}) for some fixed radius $R$ implies that $D_R$ is a
cone in $N$.
\end{theorem}

\begin{proof}[Proof of Theorem \ref{thm2.2}]

The proof of this Theorem follows the lines of the Theorem \ref{thm2.1}, and the same scheme as the proof of
Theorem 2.1 in \cite{MP4}. In this proof, however, the sign of some crucial inequalities is reversed with respect the
proof of Theorem \ref{thm2.1}. In fact, the new
geometric setting given by the {\em comparison constellation
bounded from above} give us inequality (\ref{lemaJK2})
 so when we compare the solution of the problem
(\ref{eqPoisson1}) with the solution $E_R$ to the
Dirichlet-Poisson equation on $D_R$, we conclude, applying too the
maximum principle, that $E_R \leq \psi$ on $D_R$, and hence, using too Proposition \ref{eqschwarz},
\begin{equation}\label{firstineq}
   \mathcal{A}_1(D_R) =\int_{D_R} E_R d\sigma \leq \int_{D_R} \psi
   d\sigma \quad \\ =\int_{B^W_{T(R)}}\psi^*
   d\tilde\sigma \quad ,
    \end{equation}
where $B^W_{T(R)}$ is the symmetrization of $D_R$ in $C^m_{w,h}$.

We must remark that as the comparison constellation is bounded
from above, we have, by virtue of Corollary \ref{corVolBound},
that

$$\Vol(B^W_{\tilde r(t)})=\Vol(D_{a(t)})\geq \Vol(B^W_{a(t)})\quad ,$$
so $\tilde r(t) \geq a(t) \,\,{\textrm{for all \,}} t$ and
\begin{equation}\label{size2}
\tilde r(0)=T(R) \geq a(0)=R \,\, .
\end{equation}

Now, following the lines of the proof of Theorem \ref{thm2.1}
 we have the following, which will be proved below:
 \begin{proposition}\label{lemineq}
 \begin{equation}\label{ineqderivatives2}
 \psi^{*'}(\tilde r) \geq E^{W'}_{T(R)}(\tilde r) \,\,\,{\textrm{for all \,}} \tilde r \in [0,T(R)] .
 \end{equation}
 \end{proposition}
\vspace{0.4cm}
Since
 $$\psi^*(T(R))=\phi(T(R))=0=E^W_{T(R)}(T(R))\quad ,$$
 we obtain, integrating (\ref{ineqderivatives2}) from $\tilde r$ to $T(R)$, that
$$E^{W*}_{R}(\tilde r) \leq E^W_{T(R)}(\tilde r) \,\,\,\,{\textrm{for all \,}} \tilde r \in [0,T(R)]\quad .$$
Therefore,

\begin{equation}
   \begin{aligned}
   \mathcal{A}_1(D_R) &=\int_{D_R} E_R d\sigma \leq \int_{D_R} \psi d\sigma =\int_{B^W_{T(R)}}\psi^* d\tilde\sigma\\
   & \leq \int_{B^W_{T(R)}}E^W_{T(R)}d\tilde\sigma=
   \mathcal{A}_1(B^W_{T(R)})\quad ,
   \end{aligned}
   \end{equation}
   and the Theorem is proved.
\end{proof}
\begin{proof}[Proof of Proposition \ref{lemineq}]
This proof follows the same steps as the proof of Proposition \ref{lemineqphiE}, taking into account that in
this case the comparison constellation is bounded from above, and hence, we shall use the isoperimetric inequality
(\ref{eqIsopGeneralB}) in Theorem \ref{thmIsopGeneral1}, and inequality (\ref{CompVolDown}) in Corollary
\ref{corVolBound}, inverting all inequalities.
\end{proof}

\begin{remark} The
volume of the $W$-model may be finite and we need to guarantee
that there is enough room for the symmetrization construction, because of inequality (\ref{size2}). For
this reason, we assume that the volume of the model space is
infinite. Alternatively we could assume that $W'(r)>0$, because
then the volume $\Vol(B^W_r)$ increases to $\infty$ with $r$. This
condition, however, is more restrictive.
In the setting of Theorem \ref{thm2.1}, $\Vol(D_r)=\Vol
(B^W_{T(r)})\leq \Vol(B^W_{s(r)})$ for all $r$, so we have inequality (\ref{size1}), and  the existence
of $T(R)$ is guaranteed without any additional hypothesis on the
volume of the model space.
\end{remark}
\bigskip
\section{Intrinsic Versions}

In this section we consider the intrinsic versions of Theorems
\ref{thm2.1} and \ref{thm2.2} assuming that $P^m=N^n$. In this case, the extrinsic
distance to the pole $p$ becomes the intrinsic distance in $N$,
so, for all $r$ the extrinsic domains $D_r$ become the geodesic
balls $B^N_r$ of the ambient manifold $N$. Then, for all $x \in P$
\begin{eqnarray*}
\nabla^P r(x)&=&\nabla^N r (x),\\
H_P(x)&=&0.
\end{eqnarray*}
As a consequence, $\|\nabla^P r\|=1$, so $g(r(x))=1$ and
$\mathcal{C}(x)=h(r(x))=0$, the stretching function becomes the
identity $s(r)=r$, $W(s(r))=w(r)$, and the isoperimetric
comparison space $C_{w,g,h}^m$ is reduced to the auxiliary model
space $M^m_w$.\\

For this intrinsic viewpoint, we have the following
isoperimetric and volume comparison inequalities.

\begin{proposition}[\cite{MP5}] Let $N^n$ denote a complete
Riemannian manifold with a pole $p$. Suppose that the $p$-radial
sectional curvatures of $N^n$ are bounded from below by the
$p_w$-radial sectional curvatures of a $w$-model space $M^n_w$.
Then, for all $R>0$
\begin{displaymath}
\frac{\Vol(\partial B^N_R)}{\Vol( B^N_R)} \leq \frac{\Vol(\partial
B^w_R)}{\Vol( B^w_R)}.
\end{displaymath}

Furthermore,

\begin{equation}\label{compvolintri}
\Vol(B^N_R)\leq \Vol (B^w_R).
\end{equation}
\end{proposition}
\begin{theorem} \label{thm2.1intrinsic}
Let $B^N_R$ be a geodesic ball of a complete Riemannian manifold
$N^n$ with a pole $p$ and suppose that the $p$-radial
sectional curvatures of $N^n$ are bounded from below by the
$p_w$-radial sectional curvatures of a $w$-model space $M^n_w$. Assume that $M^n_w$ is
balanced from above. Then
\begin{equation} \label{eqMain4}
\mathcal{A}_1(B^N_R) \geq \mathcal{A}_1(B^{w}_{T(R)}) \quad ,
\end{equation}
where $B^{w}_{T(R)}$ is the Schwarz symmetrization of $B^N_R$ in
the $w$-space $M^n_w$, i.e., it is the geodesic ball in $M^n_w$
such that $\Vol(B^N_R)= \Vol(B^{w}_{T(R)})$.

Equality in (\ref{eqMain4}) for some fixed radius $R$ implies that $T(R)=R$ and that $B^N_R$ and $B^w_R$ are
isometric.
\end{theorem}
\begin{proof}
The proof follows the ideas of Theorem \ref{thm2.1}. In this case,
since, $g(r)=1$ and $h(r)=0$, the second order differential
operator $L$ agrees with the Laplacian on functions of one
variable defined on the model spaces $M^n_w$,
\begin{displaymath}
L f(r)=f''(r)+(n-1)\eta_w(r)f'(r).
\end{displaymath}
Solving the corresponding problem (\ref{eqPoisson1}) on $[0,R]$
under this conditions, transplanting the solution $\psi(r)$ to the
geodesic ball $B^N_R$, and applying Laplacian comparison analysis, (namely, using inequality (\ref{eqLap1}) in
Theorem \ref{corLapComp} when $\psi' \leq 0$),
we obtain the inequality
\begin{equation}
\Delta^N \psi(r(x))\geq -1=\Delta^N E_R(x).
\end{equation}
Since $\|\nabla^P r\|=1$, the sign of $\psi''(r)-\psi'(r)
\eta_w(r)$ is obsolete in this setting and we do not need to
assume that $M^n_w$ is $w$-balanced from below.

Therefore, since $\psi(R)=E_R(R)=0$, the Maximum Principle, gives
\begin{equation}\label{deslaplaintri}
E_R(x)\geq \psi(r(x)) \quad \textrm{for all}\quad x\in B^N_R,
\end{equation}
and we have
\begin{displaymath}
\mathcal{A}_1(B^N_R)=\int_{B^N_R} E_R d\sigma \geq \int_{B^N_R}
\psi d\sigma=\int_{B^w_{T(R)}} \psi^* d\tilde\sigma,
\end{displaymath}
where $B^w_{T(R)}$ is the Schwarz symmetrization of the geodesic
ball $B^N_R$ in the $w$-model space $M^n_w$, that is, the geodesic
ball satisfying that $\Vol(B^N_R)=\Vol(B^w_{T(R)})$. From
(\ref{compvolintri}), we know that $T(R)\leq R$.

Now, we consider the radial solution $E^w_{T(R)}(r)$ of the
problem

\begin{eqnarray*}
\Delta^{M^n_w} E&=&-1\quad \textrm{on } B^w_{T(R)},\\
E\vert_{\partial B^w_{T(R)}}&=&0.
\end{eqnarray*}

With an argument analogous to that of Theorem \ref{thm2.1}, we
conclude that
\begin{displaymath}
\psi^*(t)\geq E^w_{T(R)}(t) \quad \textrm{for all } t \in
[0,T(R)],
\end{displaymath}
and then
\begin{displaymath}
\mathcal{A}_1(B^N_R) \geq \int_{B^w_{T(R)}} \psi^* d\tilde\sigma
\geq \int_{B^w_{T(R)}} E^w_{T(R)}
d\tilde\sigma=\mathcal{A}_1(B^w_{T(R)}).
\end{displaymath}

To prove the equality assertion, we must take into account that equality in (\ref{eqMain4}) for some fixed radius
$R>0$ implies equality in (\ref{deslaplaintri}) for all $x \in B^N_R$. Then, the exponential map from the pole
$p$ generates an isometry from $B^N_R$ onto $B^w_R$ in the way described in \cite{MP4}.

\end{proof}

The following intrinsic version of Theorem \ref{thm2.2} was stated and proved in
\cite{MP4}.

\begin{theorem} \label{thm2.2intrinsic}
Let $B^N_R$ be a geodesic ball of a complete Riemannian manifold
$N^n$ with a pole $p$ and suppose that the $p$-radial
sectional curvatures of $N^n$ are bounded from above by the
$p_w$-radial sectional curvatures of a $w$-model space $M^n_w$. Assume that $M^n_w$
is totally balanced. Then
\begin{equation} \label{eqMain5}
\mathcal{A}_1(B^N_R) \leq \mathcal{A}_1(B^{w}_{T(R)}) \quad ,
\end{equation}
where $B^{w}_{T(R)}$ is the Schwarz symmetrization of $B^N_R$ in
the $w$-space $M^n_w$, i.e., it is the geodesic ball in $M^n_w$
such that $\Vol(B^N_R)= \Vol(B^{w}_{T(R)})$.

Equality in (\ref{eqMain5}) for some fixed radius $R$ implies that $T(R)=R$ and that $B^N_R$ and
$B^w_R$ are isometric.
\end{theorem}
\begin{proof}

We solve (\ref{eqPoisson1}) under the same conditions as in the proof of Theorem \ref{thm2.1intrinsic}, and
transplant the solution to the geodesic ball $B^N_R$. In this case, the $p$-radial sectional curvatures of $N$
are bounded from above by the $p_w$-radial sectional curvatures in $M^n_w$, and $\psi'(r) \leq 0$ so we
have the inequality
\begin{equation}\label{intrinsicLapIneq2}
\Delta^N \psi(r(x))\leq -1=\Delta^N E_R(x).
\end{equation}
Hence, $E_R \leq \psi$ on $B^N_R$ and we have inequality (\ref{eqMain}) using the same arguments as in the proof of
Theorem \ref{thm2.1intrinsic}.

The equality assertion follows from same considerations than in Theorem \ref{thm2.1intrinsic}.

\end{proof}

\begin{remark} Although we do not need the condition that the $w$-model
space be balanced from below to conclude that $E_R \leq \psi$ on
$B^N_R$, we need to guarantee that there is enough room for the
symmetrization construction. In this setting, $\Vol(B^N_R)\geq
\Vol(B^w_R)$ for each $R$, and the volume of the $w$-model may be
finite. However, if the $w$-model space is $w$-balanced from
below, $w'(r)>0$ and the volume $\Vol(B^w_r)$ increases to
$\infty$. For this reason, we assume that $M^n_w$ is totally
balanced in Theorem \ref{thm2.2intrinsic}.
\end{remark}

\section{Average mean exit time function}

The geometric average mean exit time from the extrinsic balls $D_R$, defined by the quotient
$\mathcal A_1(D_R)/\Vol (D_R)$, was introduced in \cite {MP4}, with the purpose to give some idea about
the volume-relative swiftness of the Brownian motion defined on the submanifold $P$ at infinity, in connection with
the more classical properties like transience and recurrence.\\

As alluded to in the Introduction, we have been inspired partially by the works \cite{BBC} and \cite{BG}, where
the authors find upper bounds for the torsional rigidity of domains in  Euclidean spaces which satisfy Hardy
inequalities. These inequalities guarantee that the boundaries of the domains are not too thin so that the
Brownian diffusion is guaranteed sufficient room for escape.\\

In our present setting, the thickness of the boundary is replaced by the isoperimetric inequalities
(\ref{eqIsopGeneralA}) and (\ref{eqIsopGeneralB}), satisfied by our extrinsic domains in different curvature contexts,
which controls whether the Brownian diffusion process is slow or  fast at infinity.\\

For example, although Brownian diffusion is known to be transient in Euclidean spaces of dimensions larger than $2$,
it is not sufficiently swift, however, to give even a finite average of the mean exit time at infinity for geodesic
balls, (see \cite[Corollary 5.2]{MP4}). Concerning this observation, we gave in \cite{MP4} a set of curvature
restrictions which give finiteness of the average mean exit time at infinity for minimal submanifolds. We shall
present in Corollary \ref{cor2}, a generalization of this result for submanifolds with controlled radial mean
curvature. On the other hand, in Corollary \ref{cor1} a {\em dual} version of this result is presented in the sense
that we find a set of curvature bounds which guarantee that the average of the mean exit time at infinity is
infinite, (thus obtaining a set of curvature restrictions under which the Brownian diffusion process defined
on the submanifold is {\em slow}).\\

In the following results, we shall denote as $\mathbb{\bar R}$ the {\em extended} real line so
that $\mathbb{\bar R}_{+}=\mathbb{R}_{+}\cup\{\infty\}$.

\begin{proposition}\label{average}
Let $M^m_w$ be a $w$-model space with infinite volume. Let us suppose that the following limit exists:
$$\lim_{R \to \infty} q_W(R)=q_W(\infty) \in \mathbb{\bar R}_{+}\quad .$$

Then the average mean exit time from the $R$-balls
in these model spaces satisfies:

\begin{equation}
\lim_{R\rightarrow \infty}
\frac{\mathcal{A}_1(B^w_R)}{\Vol(B^w_R)}=q_W^{2}(\infty) \in
\mathbb{\bar R}_{+}\quad .
\end{equation}
\end{proposition}
\begin{proof}

We apply L'Hospital's Rule to the differentiable functions in
$]0,\infty[$, $f(R)=\mathcal{A}_1(B^w_R)$ and $g(R)=\Vol(B^w_R)$.
Using the fact that, in the model spaces,  the derivative of the
volume of the geodesic balls is equal to the volume of the
geodesic spheres, see (\ref{eqDefq}), and equation
(\ref{derivative2}) we have

\begin{equation}\label{averagepositive}
\lim_{R\to \infty} \frac{\mathcal{A}_1(B^w_R)}{\Vol(B^w_R)} =
\lim_{R\to \infty} q^2_W(R)= \left(\lim_{R\to \infty}
q_W(R)\right)^2 \quad .
\end{equation}
\end{proof}
\begin{remark}
If $q^2_W(\infty) >0$, then, since the volume of the space is infinite,
$\lim_{R\to \infty} \Vol(B^w_R)=\infty$ and from
inequality (\ref{averagepositive}),
 $\lim_{R\rightarrow \infty}
\mathcal{A}_1(B^w_R)=\infty$.
\end{remark}

As corollaries of Proposition \ref{average} and Theorems \ref{thm2.1} and \ref{thm2.2}, we have the following
results:

\begin{corollary}\label{cor1}
Let $\{ N^{n}, P^{m}, C_{w,g,h}^{m} \}$ denote a comparison
constellation  boun\-ded from below. Assume that
$M^m_W=C_{w,g,h}^{m}$ is $w$-balanced from below, that it is
$W-$balanced from above, and that it has infinite volume. Let
$D_R$ be an extrinsic $R$-ball in $P^m$, with center at a point $p
\in P$ which also serves as a pole in $N$. If the volume of the
submanifold $P$ is infinite, and $\lim_{R \to \infty}
q_W(R)=q_W(\infty) =\infty$ then
\begin{equation}
\lim_{R\rightarrow \infty}
\frac{\mathcal{A}_1(D_R)}{\Vol(D_R)}\geq q^2_W(\infty)
=\infty\quad .
\end{equation}
\end{corollary}
\begin{proof}
Applying Theorem \ref{thm2.1},
\begin{equation}
\lim_{R\rightarrow
\infty}\frac{\mathcal{A}_1(D_R)}{\Vol(D_R)}\geq
\lim_{R\rightarrow
\infty}\frac{\mathcal{A}_1(B^W_{T(R)})}{\Vol(B^W_{T(R)})},
\end{equation}
where $B^W_{T(R)}$ is the Schwarz symmetrization of $D_R$ in the
model space $M^m_W$.

Now, suppose that $\lim_{R\rightarrow \infty}
T(R)=T_\infty<\infty$. Then,
\begin{equation}\label{infinitevolume}
\Vol(P)=\lim_{R\rightarrow \infty} \Vol(D_R)=\lim_{R\rightarrow
\infty}\Vol (B^W_{T(R)})=\Vol(B^W_{T_\infty}) <\infty,
\end{equation}
which leads to a contradiction. As a consequence, $T(R)$ goes to
$\infty$ and we can replace $T(R)$ by $R$ in the limit
construction in the model space, that is
\begin{equation}
\lim_{R\rightarrow
\infty}\frac{\mathcal{A}_1(D_R)}{\Vol(D_R)}\geq
\lim_{R\rightarrow
\infty}\frac{\mathcal{A}_1(B^W_{R})}{\Vol(B^W_{R})}\,\, .
\end{equation}

The result follows now applying Proposition \ref{average}.
To do that, we must check that
\begin{equation}
\lim_{R \to \infty}\Vol(B^W_R)=\infty ,
\end{equation}
but this follows from $\Vol(M^m_W)=\infty$.
On the other hand, we assume that $\lim_{R\to\infty} q_W(R)=\infty \in \mathbb{\bar R}$.
\end{proof}

\begin{corollary}(see \cite[Corollary 2.3]{MP4})\label{cor2}
Let $\{ N^{n}, P^{m}, C_{w,1,h}^{m} \}$ denote a comparison
constellation  bounded from above. Assume that $M^m_W=C_{w,1,h}^{m}$ is
$w$-balanced from below, is $W-$balanced from above, and has infinite volume. Let $D_R$ be an extrinsic
$R$-ball in $P^m$,
 with center at a point $p \in P$ which also serves as a pole in $N$.
 Suppose that the model space geodesic spheres do not have $0$ as a limit for their mean
 curvatures $\eta_W(R)$ as $R \to \infty$ and that these mean curvatures satisfy
 $\eta_W(r)>0 \,\,\,{\textrm{for all \,}} r>0$.

 Then $\lim_{R \to \infty} q_W(R)=q_W(\infty) <\infty$
and
\begin{equation}
\lim_{R\rightarrow \infty}
\frac{\mathcal{A}_1(D_R)}{\Vol(D_R)}\leq q^2_W(\infty)
<\infty\quad .
\end{equation}
\end{corollary}
\begin{proof}
Proceeding as in Corollary \ref{cor1}, and applying Theorem
\ref{thm2.2}, we have firstly
\begin{equation}
\lim_{R\rightarrow
\infty}\frac{\mathcal{A}_1(D_R)}{\Vol(D_R)}\leq
\lim_{R\rightarrow
\infty}\frac{\mathcal{A}_1(B^W_{T(R)})}{\Vol(B^W_{T(R)})},
\end{equation}
where $B^W_{T(R)}$ is the Schwarz--symmetrization of $D_R$ in the
model space $M^m_W$.

As in the proof of Corollary \ref{cor1}, we can replace $T(R)$ by $R$ in the limit
construction in the model space, that is
\begin{equation}
\lim_{R\rightarrow
\infty}\frac{\mathcal{A}_1(D_R)}{\Vol(D_R)}\leq
\lim_{R\rightarrow
\infty}\frac{\mathcal{A}_1(B^W_{R})}{\Vol(B^W_{R})}=\lim_{R \to \infty} q_W(R) ,
\end{equation}
and we apply Proposition \ref{average}, because by hypothesis, $\Vol(M^m_W)=\infty$, and, on the other hand, the
limit
\begin{equation}\label{finitelimit}
\lim_{R\to\infty}q_W(R) <\infty .
\end{equation}

To see inequality (\ref{finitelimit}) we use the fact that $\lim_{R\to\infty}\eta_W(R)\neq 0$.
Then, as $q_W'(R) \geq 0$, we have that $q_W(R)\eta_W(R) \leq \frac{1}{m-1}$, see \cite[Observation 3.8]{MP4}, so,
as $\eta_W(R) \geq 0$ for all $R$, we get
\begin{equation}
q_W(R) \leq \frac{1}{(m-1)\eta_W(R)} .
\end{equation}
\end{proof}
\begin{remark}
When $P$ is minimal, we may use $ h = 0$ as a bound for the $p$-radial mean curvature, and hence, since by
hypothesis $g=1$, we have: $W=w$. In this case and by virtue of the balance conditions, the model space $M^m_w$ is
totally balanced and then we have  $\eta_w(R) >0\,\,\,{\textrm{for all \,}} R>0$. Therefore, Corollary \ref{cor2}
clearly generalizes \cite[Corollary 2.3 ]{MP4}.
\end{remark}

\end{document}